\documentclass{amsart}%
\usepackage{amsfonts}
\usepackage{amsmath}
\usepackage{amssymb}
\usepackage{graphicx}%
\providecommand{\U}[1]{\protect\rule{.1in}{.1in}}
\newtheorem{theorem}{Theorem}
\theoremstyle{plain}

\newtheorem{corollary}{Corollary}

\newtheorem{definition}{Definition}
\newtheorem{example}{Example}

\newtheorem{lemma}{Lemma}
\newtheorem{notation}{Notation}

\newtheorem{proposition}{Proposition}
\newtheorem{remark}{Remark}

\numberwithin{equation}{section}
\begin{document}
\title{Ricci curvature and metric in causal spacetimes.}
\author{Javier Lafuente L\'{o}pez}
\address{Dpto. de \'{A}lgebra, Geometr\'{\i}a y Topolog\'{\i}a.\\
Universidad Complutense de Madrid. \\
Avda. Complutense s/n Madrid 28040 Spain}
\email{jlafuent@ucm.es}
\urladdr{https://blogs.mat.ucm.es/jlafuent/}
\maketitle

\begin{abstract}
A viable spacetime is one that admits a complete timelike geodesic. It is
shown that a causal diffeomorphism preserving the Ricci tensor between two
spacetimes is necessarily a homothety, if one of them is viable.

\end{abstract}

\section{Introduction.}

A metric $g\in\mathcal{C}$ of a causal spacetime $\left(  M,\mathcal{C}%
\right)  $ determines an energy tensor
\begin{equation}
T=\frac{1}{4\pi G}\left(  Ric-\frac{1}{2}Scg\right)  \label{Einst}%
\end{equation}
where $Ric=Ric_{g}$ is the Ricci tensor and $Sc=Sc_{g}$ is the scalar
curvature tensor of $g$. \ We will then say that $T$ is in the causal
structure $\mathcal{C}$. It is well known that in general an (energy) tensor
$T$ does not automatically determine the existence of the metric $g$ that
satisfies equation (\ref{Einst}), and in the event that it exists, uniqueness
is not guaranteed unless certain initial conditions are imposed. We now ask
whether the field equations (\ref{Einst}) given by the energy tensor $T$ in
the causal structure $\mathcal{C}$, have a unique solution (up to homotheties)
within their conformal class. The main result of this article is to give an
affirmative answer to this question, when the energy tensor comes from a
viable spacetime, that is, with an inertial observer of infinite life
according to their proper time. 

There is a theorem by Hawking and Penrose which ensures that if certain
physically reasonable conditions are imposed on the energy tensor $T$, the
metrics $g$ that are solutions of (\ref{Einst}) necessarily contain an
inertial observer of finite life (\cite{Pen}\cite{Hawk1}\cite{Hawk2}). It
would be wonderful to be able to ensure that there is also another with
infinite life capable of bypassing all possible singularities. For example,
this happens in the Schwarzschild spacetime, which is therefore viable. Our
result therefore ensures that the Schwarzschild metric is the only solution
(up to homothety) of equation (\ref{Einst}) with $T=0$, within its causal
structure. 

From another point of view, we have also proved that a causal diffeomorphism
preserving the Ricci tensor between two spacetimes is necessarily a homothety,
if we assume that one of them is viable. Note that the preservation of the
Ricci tensor by a conformal diffeomorphism between two Ricci-flat spaces is
automatic, and in particular a conformal diffeomorphism in Schwarzschild space
is necessarily a homothety. There is abundant bibliography regarding all this,
in the context of general semi-Riemannian manifolds. Kulkarni in
\cite{Kul1}\cite{Kul} raises the problem of determining conditions for the
curvature to determine the metric in the Riemannian case. In \cite{Turk} De
Turk investigates whether the Ricci tensor uniquely determines (except
homothety) the Riemannian structure, and proves that it is true for certain
Einstein metrics . On the other hand, Xingwang \cite{Xing}, analyzes this
uniqueness within a conformal class, and proves that, for connected, compact
and oriented Riemannian manifold of dimension $\geq2$. Finally K\"{u}hnel
\cite{Kun} \ proves this same theorem when $(M,g)$ is geodesically complete.
We have also proved this same result here using a different method (see Remark
\ref{Rade} on page \pageref{Rade}).

A semi-Riemannian manifold $\left(  M,g\right)  $ has been fixed. This means
that, $M$ is a real connected differentiable manifold with finite dimension
$m>2$ and $g$ is a metric. When $g$ is a Lorentz metric we will say that
$\left(  M,g\right)  $ is a spacetime. Normally the signature is
$(-,+,...,+)$, but we will make no distinction between a metric $g$ and
another homothetic one $cg$ with $c\in\mathbb{R}\backslash\left\{  0\right\}
$, since both define the same geometry (and in particular the same connection)
with different units of measurement. Thus we will avoid the tag 'up to
homotheties' when we claim uniqueness.

\subsection{General assumptions.}

We denote by $\nabla$ $:\left(  X,Y\right)  \mapsto\nabla_{X}Y$ the
Levi-Civita connection of $g$ , by $\langle X,Y\rangle$ to $g\left(
X,Y\right)  $, and \textrm{grad} is its gradient operator. All connections
will be assumed symmetric.

We adopt the following definition of the curvature tensor $R$:
\[
R(X,Y)Z=\nabla_{X}\nabla_{Y}Z-\nabla_{Y}\nabla_{X}Z+\nabla_{\lbrack
X,Y]}Z\quad\text{for }X,Y,Z\in\mathfrak{X}(M)
\]
and the Ricci tensor is defined by:
\[
\mathrm{Ric}(X,Y)=\mathrm{Trace}(\mathfrak{X}(M)\ni V\longrightarrow
R(X,V)Y)\quad\text{for }X,Y\in\mathfrak{X}(M)
\]
If $\left(  X,Y\right)  \rightarrow T\left(  X,Y\right)  $ is a covariant
tensor we denote by $X\mapsto\widetilde{T}\left(  X\right)  $ the unique
$\left(  1,1\right)  $ tensor such that
\[
\left\langle \widetilde{T}\left(  X\right)  ,Y\right\rangle =T\left(
X,Y\right)
\]

\section{Conformal Structure.}

The metrics $g$, $\overline{g}$ of $M$ are said to be conformally equivalent
if there exists a differentiable function $\sigma:M\longrightarrow\mathbb{R}$,
with $\overline{g}=e^{2\sigma}g$. This is an equivalence relation, and we will
denote by $\mathcal{C}=[g]$ the equivalence class defined by the metric $g$.
$\mathcal{C}$ is called a conformal structure on $M$.

Note that two Lorentzian metrics on $M$ define the same causal structure if
and only if they define the same light cones in each tangent space or
equivalently define the same causality relation on $M$. For this reason we
also call $\left(  M,\mathcal{C}\right)  $ a causal space and we will assume
it is time-oriented.

\begin{definition}
\label{viable} The causal space $\left(  M,\mathcal{C}\right)  $ is said to be
viable, if there exists a metric $g\in\mathcal{C}$, such that $\left(
M,g\right)  $ is viable, in the sense that it has a timelike geodesic
$\gamma=\gamma\left(  t\right)  $ defined $\forall t>0$.
\end{definition}

\subsection{Conformal connections.}

A connection $\overline{\nabla}$ on $\left(  M,\mathcal{C}\right)  $, is said
to be conformal if it preserves the conformal structure under parallel
transport. That is, if parallel transport induces conformal (linear)
transformations between the tangent spaces. This happens in the Lorentz case,
if and only if $\overline{\nabla}$ preserves the light cones under parallel
transport, and we have the following result:

\begin{theorem}
[\cite{Kul1}]\label{0.3.1}The necessary and sufficient condition for a
connection $\overline{\nabla}$ on $M$ to be conformal is that there exists a
vector field $A\in\mathfrak{X}(M)$ such that:
\begin{equation}
\overline{\nabla}_{X}Y-\nabla_{X}Y=\left\langle A,X\right\rangle
Y+\left\langle A,Y\right\rangle X-\left\langle X,Y\right\rangle A\quad
\text{for all }X,Y\in\mathfrak{X}(M) \label{concon}%
\end{equation}
In particular, $A=(\mathrm{grad})_{g}\sigma$, if and only if $\nabla$ and
$\overline{\nabla}$ are the Levi-Civita connections associated with $g$ and
$\overline{g}=e^{2\sigma}g$ respectively.
\end{theorem}

\begin{remark}
\label{0.3.2}The condition $A=\mathrm{grad}\sigma$ is equivalent to
$\alpha=d\sigma$, where $\alpha$ is the 1-form $g$-metrically equivalent to
$A$ (i.e., $\alpha:\mathfrak{X}(M)\ni X\longrightarrow\langle A,X\rangle
\in\mathbb{R}$). In particular the condition $d\alpha=0$, means that
$\overline{\nabla}$ is locally metric, that is: In a neighborhood $U$ of each
point of $M$, there exists a differentiable $\sigma:U\longrightarrow
\mathbb{R}$ such that $\overline{\nabla}$ is the Levi-Civita connection of
$e^{2\sigma}g$.
\end{remark}

\begin{corollary}
The null geodesic lines of two conformal connections necessarily coincide.
\end{corollary}

\subsection{Comparison of Ricci tensors.}

\begin{notation}
\label{notacion}In what follows, the symbols $\nabla$, $\overline{\nabla}$ and
$A$ have the same meaning as in theorem \ref{0.3.1}. The reference metric is
$g$ with Ricci tensor $\mathrm{Ric}$, while $\overline{\mathrm{Ric}}$ will be
the corresponding tensor of $\overline{\nabla}$. If $\alpha$ is the unique
1-form such that $\widetilde{\alpha}=A$ (that is $\alpha\left(  X\right)
=\left\langle A,X\right\rangle $), we define $Q=\nabla\alpha-\alpha
\otimes\alpha$ therefore
\begin{equation}
\tilde{Q}(X)=\nabla_{X}A-\langle A,X\rangle A \label{Qtilde}%
\end{equation}
and we then have the following result (see \cite{Kul1}):
\end{notation}

\begin{theorem}
\label{0.5.1}Assuming $\overline{\mathrm{Ric}}$ is symmetric, the difference
tensor $\mathrm{E}=\overline{\mathrm{Ric}}-\mathrm{Ric}$ satisfies the
equation%
\begin{equation}
-\frac{1}{m-2}\mathrm{E}+\frac{1}{2\left(  m-2\right)  \left(  m-1\right)
}tr\left(  \widetilde{\mathrm{E}}\right)  g=Q+\frac{1}{2}\left\langle
A,A\right\rangle g \label{mainE}%
\end{equation}

\end{theorem}

\begin{proof}
The difference tensor
\[
E\left(  X,Z\right)  =\overline{Ric}\left(  X,Z\right)  -Ric\left(
X,Z\right)  =tr\left\{  Y\rightarrow D\left(  X,Y\right)  Z\right\}
\]
is given by the formula%
\[
E\left(  X,Z\right)  =\left(  2-m\right)  Q\left(  X,Z\right)  -\left\{
q+\left(  m-1\right)  \left\langle A,A\right\rangle \right\}  \left\langle
X,Z\right\rangle
\]
where $q=tr\left(  \widetilde{Q}\right)  $ that is to say%
\begin{equation}
E=\left(  2-m\right)  Q-\left\{  \left(  q+\left(  m-1\right)  \right)
\left\langle A,A\right\rangle \right\}  g \label{E}%
\end{equation}

Using (\ref{E}) , the tensor $\widetilde{E}\left(  X\right)  =\widetilde
{\overline{Ric}}\left(  X\right)  -\widetilde{Ric}\left(  X\right)  $ has the
following expression:%
\[
\widetilde{E}\left(  X\right)  =\left(  2-m\right)  \widetilde{Q}\left(
X\right)  -\left\{  q+\left(  m-1\right)  \left\langle A,A\right\rangle
\right\}  X
\]
contracting $\widetilde{E}$ yields
\begin{align}
tr\left(  \widetilde{E}\right)   &  =\left(  2-m\right)  q-\left\{  q+\left(
m-1\right)  \left\langle A,A\right\rangle \right\}  m\label{trE}\\
&  =\left(  2-2m\right)  q-\left(  m-1\right)  m\left\langle A,A\right\rangle
\nonumber
\end{align}
solving for $q$ in (\ref{trE})%
\begin{equation}
q=\frac{1}{2-2m}tr\left(  \widetilde{E}\right)  -\frac{m}{2}\left\langle
A,A\right\rangle \label{q}%
\end{equation}
and solving for $Q$ in (\ref{E})
\begin{equation}
Q=\frac{1}{2-m}E+\frac{1}{2-m}\left\{  q+\left(  m-1\right)  \left\langle
A,A\right\rangle \right\}  g \label{Q}%
\end{equation}
finally substituting (\ref{q}) into (\ref{Q}) yields (\ref{mainE})
\end{proof}

With notation \ref{notacion}, we have the following result

\begin{proposition}
\label{0.5.3} The condition $\overline{\mathrm{Ric}}=\mathrm{Ric}$ is
equivalent to saying that the field $A$ satisfies%
\begin{equation}
\nabla_{X}A=\langle A,X\rangle A-\frac{1}{2}\langle A,A\rangle X,\quad\forall
X\in\mathfrak{X}(M)\label{ATP}%
\end{equation}

\end{proposition}

\begin{proof}
If $\mathrm{E}=\overline{\mathrm{Ric}}-\mathrm{Ric}=0$, by (\ref{mainE}) it
follows that $Q+\frac{1}{2}\langle A,A\rangle g=0$, that is
\[
\widetilde{Q}+\frac{1}{2}\langle A,A\rangle\widetilde{g}=0
\]
since $\widetilde{g}=id$, by (\ref{Qtilde}) we have
\[
\underset{\widetilde{Q}\left(  X\right)  }{\underbrace{\nabla_{X}A-\langle
A,X\rangle A}}+\frac{1}{2}\langle A,A\rangle X=0
\]
and thus $A$ satisfies (\ref{ATP}). Conversely, \ if $A$ satisfies (\ref{ATP})
then $\widetilde{Q}\left(  X\right)  =-\frac{1}{2}\langle A,A\rangle X$ and
therefore $Q+\frac{1}{2}\langle A,A\rangle g=0$. By (\ref{mainE}) it is seen
that%
\begin{equation}
-\frac{1}{m-2}\mathrm{E}+\frac{1}{2\left(  m-2\right)  \left(  m-1\right)
}tr\left(  \widetilde{\mathrm{E}}\right)  g=0\label{EtrE}%
\end{equation}
but then
\begin{equation}
\mathrm{E}=\frac{1}{2\left(  m-1\right)  }\mathrm{Trace}\left(  \widetilde
{\mathrm{E}}\right)  g\label{EEtilde}%
\end{equation}
and therefore since $\widetilde{g}=Id$
\begin{align*}
\widetilde{\mathrm{E}} &  =\frac{1}{2\left(  m-1\right)  }\mathrm{Trace}%
\left(  \widetilde{\mathrm{E}}\right)  Id\Rightarrow\\
\mathrm{Trace}\left(  \widetilde{\mathrm{E}}\right)   &  =\frac{m}{2\left(
m-1\right)  }\mathrm{Trace}\left(  \widetilde{\mathrm{E}}\right)  \\
&  \Rightarrow0=-\frac{m-2}{2\left(  m-1\right)  }\mathrm{Trace}\left(
\widetilde{\mathrm{E}}\right)
\end{align*}
If there exists $p$ such that $\left.  \mathrm{Trace}\left(  \widetilde
{\mathrm{E}}\right)  \right\vert _{p}\neq0$, then $m=2$, but $m\geq3$, by
hypothesis. By (\ref{EtrE}) it is $\mathrm{E}=0.$
\end{proof}

\subsection{Atypical fields.}

Motivated by proposition \ref{0.5.3} we establish the following definition

\begin{definition}
\label{0.5.2} A field $A\in\mathfrak{X}(M)$ will be called atypical, briefly
ATP, if it satisfies:
\begin{equation}
\nabla_{X}A=\langle A,X\rangle A-\frac{1}{2}\langle A,A\rangle X,\quad\forall
X\in\mathfrak{X}(M) \label{atp}%
\end{equation}
Note that the identically zero field $A$ is atypical. On the other hand, if
$\left(  X_{1},\ldots,X_{m}\right)  $ is a parallelization then $A$ is
atypical if and only if (\ref{atp}) is satisfied for each $X=X_{i}$.
\end{definition}

Now, in the context of Notation \ref{notacion} Theorem \ref{0.5.3} the
condition $\overline{\mathrm{Ric}}=\mathrm{Ric}$ is equivalent to the field
$A$ being atypical. Atypical fields are generally difficult to find, we show
an example:

\begin{example}
\label{Ex1}with $M=\left\{  \left(  x,y\right)  \in\mathbb{R}^{2}%
:y>x>0\right\}  \backslash\ $ and the Minkowski metric $g=-dx^{2}+dy^{2}$,
which in hyperbolic polar coordinates $\left(  \rho,\theta\right)  $
$x=\rho\sinh\theta,$ $y=\rho\cosh\theta$ is written $g=d\rho^{2}-\rho
^{2}d\theta^{2}$, we observe that there is a conformal metric
\[
\overline{g}=\frac{1}{\rho^{4}}g
\]
with $\mathrm{Ric}=\overline{\mathrm{Ric}}=0$, as a consequence of the
conformal character of hyperbolic inversions. Indeed

The inversion $\iota$ of the unit hyperboloid is written in polar coordinates
\[
\left(  \overline{\rho},\overline{\theta}\right)  \overset{\iota
}{\longleftarrow}\left(  \rho,\theta\right)  :\left\{
\begin{tabular}
[c]{l}%
$\rho=1/\overline{\rho}$\\
$\theta=\overline{\theta}$%
\end{tabular}
\ \ \ \ \ \right.  ,\left\{
\begin{tabular}
[c]{l}%
$d\rho^{2}=-\left(  1/\overline{\rho}^{4}\right)  d\overline{\rho}^{2}$\\
$d\theta^{2}=d\overline{\theta}^{2}$%
\end{tabular}
\ \ \ \ \ \right.
\]
and
\begin{align*}
\iota^{\ast}g  &  =\left(  1/\overline{\rho}^{4}\right)  d\overline{\rho}%
^{2}-\left(  1/\overline{\rho}^{2}\right)  d\overline{\theta}^{2}\\
&  =\frac{1}{\overline{\rho}^{4}}\left(  d\overline{\rho}^{2}-\overline{\rho
}^{2}d\overline{\theta}^{2}\right)
\end{align*}
so indeed $\iota$ is conformal. From another point of view, in reality we have
proved that $\frac{1}{\overline{\rho}^{4}}\left(  d\overline{\rho}%
^{2}+\overline{\rho}^{2}d\overline{\theta}^{2}\right)  $ is the expression of
the metric $g$ in the coordinates $\left(  \overline{\rho}=1/\rho
,\overline{\theta}=\theta\right)  $ so formally $\overline{g}$ has (like $g$)
zero Ricci tensor. But%
\[
\frac{1}{\rho^{4}}=e^{2\sigma}\text{, with }\sigma=-2\ln\rho
\]
therefore
\[
\alpha=d\sigma=\frac{-2}{\rho}d\rho\Rightarrow\widetilde{\alpha}=A=\frac
{-2}{\rho}\frac{\partial}{\partial\rho}%
\]
although the exposed theory is not consistent for dimension $m=2$, it can be
verified that $A$ is the atypical field and that $\operatorname{div}A=0$. If
we define $\phi\left(  x_{0},\ldots,x_{m}\right)  =-x_{0}^{2}+x_{1}^{2}%
+\cdots+x_{m}^{2}$ it can be proved that there is a diffeomorphism $\Phi$ of
\[
M=\left\{  x=\left(  x_{0},\ldots,x_{m}\right)  :\phi\left(  x\right)
<0\right\}
\]
in $\mathbb{R}^{+}\times\mathbb{H}=\left\{  \left(  \rho,\theta\right)
:\rho>0\text{, }\theta\in\mathbb{H}\right\}  $ where $\mathbb{H}$ is the
hyperboloid of equation
\[
\mathbb{H}:\phi\left(  x\right)  =-1
\]
which has the equations
\[
\Phi:\left\{
\begin{tabular}
[c]{l}%
$\rho=\sqrt{-\phi\left(  x\right)  }$\\
$\theta=x/\sqrt{-\phi\left(  x\right)  }$%
\end{tabular}
\right.
\]

If we endow $M$ with the Minkowski metric $\eta=-dx_{0}^{2}+dx_{1}^{2}%
+\cdots+dx_{m}^{2}$ and $\mathbb{H}$ with the Riemannian metric it can be
proved that%
\[
\Phi:\left(  M,\eta\right)  \rightarrow\left(  \mathbb{R}^{+}\times
\mathbb{H},h=d\rho^{2}-\rho^{2}d\theta^{2}\right)
\]
is an isometry. Note that $h$ is Lorentzian but with signature $\left(
+,-\cdots-\right)  $. Therefore $\left(  \mathbb{R}^{+}\times\mathbb{H}%
,h\right)  $ is Ricci-flat, and the map $\phi:\mathbb{R}^{+}\times
\mathbb{H\rightarrow R}^{+}\times\mathbb{H}$, $\left(  \rho,\theta\right)
\mapsto\left(  \overline{\rho}=1/\rho,\overline{\theta}=\theta\right)  $ is
conformal with factor $\frac{1}{\rho^{4}}$ that is%
\[
\phi^{\ast}h=\frac{1}{\rho^{4}}h
\]
and since $\frac{1}{\rho^{4}}=e^{2\sigma}$, with $\sigma=-2\ln\rho$,
$\alpha=d\sigma=\frac{-2}{\rho}d\rho$ it is concluded that the field%
\[
\widetilde{\alpha}=A=\frac{-2}{\rho}\frac{\partial}{\partial\rho}%
\]
is atypical.
\end{example}

On the other hand these atypical fields show certain global properties that
will have their consequences in the following section.

\begin{proposition}
\label{1.2}Let $A$ be an atypical field. If there exists $p\in M$ with
$\langle A,A\rangle(p)\neq0$, then it is $\langle A,A\rangle(x)\neq0$ for all
$x\in M$.
\end{proposition}

\begin{proof}
Indeed, let $x\in M$, and $\gamma:(a,b)\longrightarrow M$ a differentiable
curve (which can be assumed an integral curve of a field $X$) with $a<0<1<b$
and $\gamma(0)=p$, $\gamma(1)=x$. Calling $f(t)=\langle A,A\rangle\circ
\gamma(t)$ and $F(t)=\langle A(\gamma(t)),\gamma^{\prime}(t)\rangle$ we
verify:
\begin{align*}
X(\langle A,A\rangle)(\gamma(t))  &  =f^{\prime}(t)=\langle A,X\rangle
(\gamma(t))\langle A,A\rangle(\gamma(t))\\
&  =F(t)f(t)\quad\forall t\in(a,b),\text{ and }f(0)\neq0
\end{align*}
therefore the hypotheses of the following lemma are satisfied, and it is
concluded then that $F(1)=\langle A,A\rangle(\gamma(1))=\langle A,A\rangle
(x)\neq0$.
\end{proof}

\begin{lemma}
Let $f:(a,b)\rightarrow\mathbb{R}$ be a differentiable function not
identically zero. Suppose there exists a continuous $F:(a,b)\rightarrow
\mathbb{R}$ such that $f^{\prime}(t)=F(t)f(t)$ $\forall t\in(a,b)$. Then $f$
is never zero (that is $f(t)\neq0$ $\forall t\in(a,b)$).
\end{lemma}

\begin{proof}
If the function $f\ $is not identically zero there exists $t_{0}\in(a,b)$ with
$f\left(  t_{0}\right)  \neq0$ and satisfies the differential equation%
\begin{align*}
\frac{dy}{dt}  &  =yF\left(  t\right)  \Rightarrow\int\frac{dy}{y}=\int F\\
&  \Rightarrow y=Ce^{\varphi\left(  t\right)  }\text{ with}\left\{
\begin{array}
[c]{c}%
\text{ }\varphi^{\prime}=F\\
C\in\mathbb{R}%
\end{array}
\right.
\end{align*}
since the function $f\ $is not identically zero it is $C\neq0$, so $f$ is
never zero.
\end{proof}

\begin{proposition}
\label{1.1}If $\mathrm{Ric}=\overline{\mathrm{Ric}}$, then $\overline{\nabla}$
is locally metric.
\end{proposition}

\begin{proof}
If $\langle A,A\rangle$ is identically zero, then since $A$ is atypical:
\[
\nabla_{X}A=\langle A,X\rangle A-\frac{1}{2}\langle A,A\rangle X=\langle
A,X\rangle A,\quad\forall X\in\mathfrak{X}(M)
\]
Thus for all $X,Y\in\mathfrak{X}(M)$ if $\alpha(X)=\langle A,X\rangle$ it is
verified:
\begin{align*}
X(\alpha(Y)) &  =X(\langle A,Y\rangle)=\langle\nabla_{X}A,Y\rangle+\langle
A,\nabla_{X}Y\rangle\\
&  =\langle A,X\rangle\langle A,Y\rangle+\langle A,\nabla_{X}Y\rangle
\end{align*}
Therefore $d\alpha=0$, (and by 0.3.2 $\overline{\nabla}$ is locally metric)
since:
\[
X(\alpha(Y))-Y(\alpha(X))=\langle A,\nabla_{X}Y-\nabla_{Y}X\rangle=\langle
A,[X,Y]\rangle
\]
If $\left.  \langle A,A\rangle\right\vert _{p}\neq0$ for some $p$ then by
proposition \ref{1.2} $\langle A,A\rangle$ is never zero. Suppose now for
example $\langle A,A\rangle>0$ on $M$. We have:
\begin{align*}
X(\langle A,A\rangle) &  =2\langle\nabla_{X}A,A\rangle=2(\langle
A,X\rangle\langle A,A\rangle-\frac{1}{2}\langle A,A\rangle\langle
X,A\rangle)\\
&  =\alpha(X)\langle A,A\rangle\quad\forall X\in\mathfrak{X}(M)
\end{align*}
Thus $X(\log(\langle A,A\rangle))=\alpha(X)$, so $\alpha=d(\log(\langle
A,A\rangle))$. Analogously if $\langle A,A\rangle<0$ on $M$ it is
$\alpha=d(\log(-\langle A,A\rangle))$, and the theorem is true on the
non-empty open submanifold $M^{\prime}=\{x\in M:\langle A,A\rangle(x)\neq0\}$.
\end{proof}

We have therefore proved the following corollary

\begin{corollary}
If $A$ is atypical and for some $p\in M$ it is $\langle A,A\rangle(p)\neq0$,
then it is $\langle A,A\rangle(x)\neq0$ for all $x\in M$. and if $\sigma
=\log(|\langle A,A\rangle|)$ it is concluded that $\overline{\nabla}$ is the
Levi-Civita connection for the metric $\overline{g}=e^{2\sigma}g\in
\mathcal{C}$. In particular $A=\mathrm{grad}\sigma$.
\end{corollary}

\section{Main results.}

\subsection{Assumptions.}

Recall: $\nabla$ is the Levi-Civita connection of the metric $g$ of $M$,
$\mathcal{C}=[g]$. $\overline{\nabla}$ is a $\mathcal{C}$-conformal symmetric
connection, with the same Ricci tensor $\overline{\mathrm{Ric}}=\mathrm{Ric}$
as $g$, and $A\in\mathfrak{X}(M)$ is the atypical field of \ref{0.3.1} that
relates $\nabla$ and $\overline{\nabla}$.

Since $M$ is connected, from proposition \ref{1.2} it follows that the
following possibilities is verified:

$\langle A,A\rangle(x)>0$ for all $x\in M$. $\langle A,A\rangle(x)<0$ for all
$x\in M$.

$\langle A,A\rangle(x)=0$ for all $x\in M$.

\begin{proposition}
\label{1.4}Suppose $\overline{\mathrm{Ric}}=\mathrm{Ric}$. Then $A$ defines a
field of $\nabla$-pregeodesics. Furthermore, if $\langle A,A\rangle\neq0$ on
$M$, all integral curves of $A$ define incomplete $\nabla$-geodesics.
Consequently: If $A$ is spacelike (respectively timelike) then there exist
spacelike (respectively, timelike) $\nabla$-geodesics that are incomplete.
\end{proposition}

\begin{proof}
There exists a field $U\in\mathfrak{X}(M)$ and a differentiable
$f:M\longrightarrow\mathbb{R}^{+}$, such that $\langle U,U\rangle=\epsilon$
and $A=fU$, where $\epsilon=1$, if $A$ is spacelike, and $\epsilon=-1$, if $A$
is timelike. On the other hand, since $\nabla_{A}A=\frac{1}{2}\langle
A,A\rangle A$ it is concluded that $A$ is a field of $\nabla$-pregeodesics,
therefore the integral curves of $U$ are $\nabla$-geodesics. But
\[
\frac{\epsilon}{2}f^{3}U=\frac{1}{2}\langle A,A\rangle A=\nabla_{A}%
A=\nabla_{fU}(fU)=f\{U(f)U+f\nabla_{U}U\}=fU(f)U
\]
dividing both members by $f^{3}$ we obtain: $\frac{U(f)}{f^{2}}=\frac
{\epsilon}{2}$. Under these conditions, all ($\nabla$-geodesics) integral
curves of $U$ are incomplete. Indeed, suppose there exists an integral curve
$\gamma:\mathbb{R}\longrightarrow M$ of $U$. The function $\varphi
=f\circ\gamma$ then satisfies the differential equation: $\frac{\varphi
^{\prime}}{\varphi^{2}}=\frac{\epsilon}{2}$ with solution:%
\[
f(\gamma(t))=\frac{2\epsilon f(\gamma(0))}{2-\epsilon f(\gamma(0))t}%
\]
in particular it is concluded that the function $f$ would not be
differentiably defined in a neighborhood of $\gamma(\frac{2}{\epsilon
f(\gamma(0))})$.
\end{proof}

\begin{remark}
[Theorem]\label{Rade} In the Riemannian case all geodesics are spacelike, and
proposition \ref{1.4} guarantees that if $\left(  M,g\right)  $ is complete
then it does not admit atypical fields, and therefore the Ricci tensor
determines the metric. This result has also been proved in \cite{Kun}, which
generalizes \cite{Xing}.
\end{remark}

\begin{lemma}
\label{1.6}If $\overline{\mathrm{Ric}}=\mathrm{Ric}$, and $\langle
A,A\rangle=0$ on $M$. Then, if $A$ is not identically zero, then $A$ is
nowhere zero and $(M,g)$ is timelike and spacelike incomplete, and furthermore
all timelike $\nabla$-geodesics are incomplete.
\end{lemma}

\begin{proof}
Since $\langle A,A\rangle=0$ it is $\nabla_{A}A=0$, and $A$ defines a field of
null $\nabla$-geodesics. Let $\gamma$ be a maximal $\nabla$-geodesic,
$\langle\gamma^{\prime},\gamma^{\prime}\rangle=\epsilon=\pm1$, and let
$E_{0}=\gamma^{\prime},E_{1},...,E_{n}$ be a parallel orthonormal basis along
$\gamma$ with $\epsilon_{i}=\langle E_{i},E_{i}\rangle=\pm1$. It is then
verified that:
\[
A(\gamma(t))=\sum a_{i}(t)\epsilon_{i}E_{i}(t)\quad\text{with }a_{i}=\langle
A,E_{i}\rangle
\]
Since $A$ is atypical, and $\langle A,A\rangle=0$ it is $\nabla_{X}A=\langle
A,X\rangle A$ and we have:
\[
\frac{\nabla A}{dt}=\sum\epsilon_{i}\frac{da_{i}}{dt}E_{i}=\langle
A,\gamma^{\prime}\rangle A=a_{0}\sum a_{i}\epsilon_{i}E_{i}%
\]
in particular, $\frac{da_{0}}{dt}=a_{0}^{2}$. If $\langle A(\gamma
(0)),\gamma^{\prime}(0)\rangle=\epsilon a_{0}(0)\neq0$, taking $\alpha
=a_{0}(0)\neq0$, it is concluded $a_{0}(t)=\frac{\epsilon\alpha}%
{-\epsilon\alpha t+1}$, $\gamma$ cannot be complete, since $a_{0}(t)$ would
become infinite for $t=\frac{1}{\epsilon\alpha}$. If $\langle A(\gamma
(0)),\gamma^{\prime}(0)\rangle=\epsilon a_{0}(0)=0$, it is concluded that
$\epsilon a_{0}(t)=\langle A,\gamma^{\prime}\rangle(t)=0$, and since $A$ is
atypical and $\langle A,A\rangle=0$, $\frac{\nabla A}{dt}=\langle
A,\gamma^{\prime}\rangle A=0$. Thus $A$ is parallel along $\gamma$. This
proves that if for a point $p\in M$, it is $A(p)=0$, then $A\circ\gamma$ is
identically zero for all non-null geodesics $\gamma$ starting from $p$.
Therefore the open set $N=\{x\in M:A(x)\neq0\}$, is also closed, and since it
is non-empty, it coincides with the complete manifold $M$. In the Lorentz
case, any maximal timelike $\nabla$-geodesic $\gamma$ verifies $\langle
A(\gamma(0)),\gamma^{\prime}(0)\rangle\neq0$, and is therefore incomplete.
\end{proof}

\begin{theorem}
[Main.]\label{1.8}If $\left(  M,g\right)  $ is a viable spacetime (Definition
\ref{viable})and $\overline{\nabla}$ is a symmetric connection with the same
Ricci tensor as $g$ that preserves the light cones by parallel transport, then
$\overline{\nabla}$ is the metric connection of $g$.
\end{theorem}

\begin{proof}
If $\nabla$ is the metric connection of $g,$since $\overline{\mathrm{Ric}%
}=\mathrm{Ric}$ , by proposition \ref{0.5.3}, there exists $A$ atypic
vectorfield verifying (\ref{atp})

If $\langle A,A\rangle=0$, and there exists $p\in M$ with $A(p)\neq0$, by
Lemma \ref{1.6} $A$ does not vanish at any point, all timelike geodesics are
incomplete which contradicts the hypothesis. Suppose then that there exists
$p\in M$ with $\langle A,A\rangle(p)\neq0$. By \ref{1.2} we have $\langle
A,A\rangle\neq0$ in all of $M$. Let $\gamma:\mathbb{R}\longrightarrow M$ be
the timelike geodesic which by hypothesis is complete, and let $(E_{0}%
=\gamma^{\prime},E_{1},...,E_{n})$ be a parallel orthonormal basis along
$\gamma$ where $A=\sum a_{i}E_{i}$ with $a_{i}:\mathbb{R}\rightarrow
\mathbb{R}$ differentiable. Note that by Lemma 1.4 $A$ defines a field of
incomplete pregeodesics, so $A\circ\gamma$ is never tangent to $\gamma$, that
is, the function:
\[
f(t)=\frac{1}{2}(a_{1}(t)^{2}+...+a_{n}(t)^{2})
\]
verifies $f(t)>0$ for all $t\in\mathbb{R}$. Since $\frac{\nabla A}{dt}%
=\sum\frac{da_{i}}{dt}E_{i}=\langle A,\gamma^{\prime}\rangle A-\frac{1}%
{2}\langle A,A\rangle\gamma^{\prime}=-a_{0}[\sum a_{i}E_{i}]-\frac{1}%
{2}(-a_{0}^{2}+a_{1}^{2}+...+a_{n}^{2})E_{0}$, and comparing the coefficients
multiplying $E_{0}$ we get:
\[
\frac{da_{0}}{dt}=-\frac{1}{2}a_{0}^{2}-f(t)
\]

Calling $y(t)=-a_{0}(t)$, we conclude that $y(t)$ is defined on all of
$\mathbb{R}$ and satisfies the differential equation:
\begin{equation}
\frac{dy}{dt}=\frac{1}{2}y^{2}+f(t) \label{EDO}%
\end{equation}
where $f:\mathbb{R}\longrightarrow\mathbb{R}$ is differentiable, with
$f(t)>0$.
\[%
\begin{tabular}
[c]{l}%
We can choose the direction of traversal of\\
$\gamma$, such that $a_{0}\left(  0\right)  =-\left.  \left\langle
A,E_{0}\right\rangle \right\vert _{t=0}>0$%
\end{tabular}
\]

This contradicts the following Lemma.
\end{proof}

\begin{lemma}
Any solution $y=y\left(  t\right)  $ of the differential equation:
\[
\frac{dy}{dt}=\frac{1}{2}y^{2}+f,\quad\text{where }f:\mathbb{R}\rightarrow
\mathbb{R}\text{ is differentiable, with }f(t)>0
\]
with $y\left(  0\right)  >0$, cannot be defined on all of $\mathbb{R}$.
\end{lemma}

\begin{proof}
Suppose there exists $\psi:\mathbb{R}^{+}\longrightarrow\mathbb{R}$ a solution
of the differential equation defined on all of $\mathbb{R}$ that satisfies
$\beta=\psi(0)>0$. Let
\[
\varphi(t)=\frac{2\beta}{-\beta t+2}
\]
which is the solution to $\frac{dy}{dt}=\frac{1}{2}y^{2}$ such that
$\varphi(0)=\beta$. Consider the interval $I=\left]  0,b=\frac{2}{\beta
}\right[  $. Note that $\varphi>0$ on $I$ and $\lim_{t\rightarrow b}
\varphi(t)=+\infty$. Let us prove that $\varphi\leq\psi$ on $I$, which
prevents the existence of $\psi(b)$. For this it is sufficient to prove that
the set $J=\{s\in I:\varphi(t)\leq\psi(t)\text{ for }0<t<s\}$ is non-empty, as
well as open and closed in $I$:

Since $\varphi(0)=\psi(0)=\beta$, $\varphi^{\prime}(0)=\frac{1}{2}\beta
^{2}<\frac{1}{2}\beta^{2}+f(0)=\psi^{\prime}(0)$, it is concluded by $(\ast)$
that for a certain $\epsilon>0$, and all $t$ with $0<t<\epsilon$, we have
$\varphi(t)<\psi(t)$ and $(0,\epsilon)\subset J$, and $J\neq\emptyset$. On the
other hand, if $s_{1}\in I-J$ there exists $t_{1}$ with $\tau<t_{1}<s_{1}$ and
$\varphi(t_{1} )>\psi(t_{1})$, and this remains true for all values
$s\in(s_{1}-\delta,s_{1}+\delta)$ with $\delta<s_{1}-t_{1}$, thus $I-J$ is
closed. Finally, let us prove that $J$ is open:

Indeed, if $s\in J$ we have $0<\varphi(s)\leq\psi(s)$. Then:
\[
\varphi^{\prime}(s)=\frac{1}{2}\varphi(s)^{2}\leq\frac{1}{2}\psi(s)^{2}%
<\frac{1}{2}\psi(s)^{2}+f(s)=\psi^{\prime}(s)
\]
by $(\ast)$ there exists $\varepsilon$ such that $\varphi(t)<\psi(t)$ for
$s<t<s+\varepsilon$, thus $(s-\varepsilon,s+\varepsilon)\subset J$. $(\ast)$
Let $f=f\left(  t\right)  ,$ $g=g\left(  t\right)  $ be real differentiable
functions defined around $\tau\in\mathbb{R}$, with $f\left(  \tau\right)
=g\left(  \tau\right)  $. Suppose that $f^{\prime}\left(  \tau\right)
<g^{\prime}\left(  \tau\right)  $. There then exists an $\varepsilon>0$ such
that $f\left(  t\right)  \leq g\left(  t\right)  $ for all $t$ with
$\tau<t<\tau+\varepsilon$.
\end{proof}

Theorem 1.8 admits other equivalent statements, which we include here as corollaries.

\begin{corollary}
A conformal diffeomorphism between spacetimes that preserves the Ricci tensor
is necessarily a homothety, if at least one of them is viable.
\end{corollary}

\begin{corollary}
In a viable spacetime, the Ricci tensor determines the metric up to a
multiplicative constant.
\end{corollary}

\section{Conclusions.}

It is a strange phenomenon that two distinct conformal non homothetic metrics
$g$ and $\overline{g}=e^{2\sigma}g$, share the same Ricci tensor, as this is
equivalent to the existence of an atypical field $A=\operatorname{grad}%
_{g}\sigma$ verifying the exotic property $\nabla_{X}A=\langle A,X\rangle
A-\frac{1}{2}\langle A,A\rangle X$ (\ref{atp}) which relates elements that
depend on the first-order derivatives of the metric with others that depend
only on the point value of the metric. In general the search for atypical
fields turns out to be a formidable problem.

But at a theoretical level, if the existence of such a field is admitted,
(that is $\mathrm{Ric}_{g}=\mathrm{Ric}_{\overline{g}}$) then we prove that it
$A$ is a pregeodesic field with incomplete integral curves in the Riemannian
case (Proposition \ref{1.4}), and this implies that the space should be
geodesically incomplete. Therefore, a complete Riemannian space does not admit
atypical fields (see note \ref{Rade})

In the Lorentz case, it is proven that an atypical field $A$ maintains the
same causal character at every point. In each of the cases (timelike,
lightlike or spacelike) it is proven that if $A$ were to coexist with a
timelike geodesic $\gamma=\gamma\left(  t\right)  $ defined for all $t$, the
function $\left\langle A,\gamma^{\prime}\right\rangle $ would be a solution to
a certain ordinary differential equation that does not admit a solution for
all $t$ (Lemma \ref{1.6} and Theorem \ref{1.8}). Thus all its timelike
geodesics are incomplete and the spacetime could not be viable. The
\emph{viable} condition has a global character.

The analysis of the existence of local atypical fields is in general an
inaccessible problem. Nevertheless, there are local obstructions to their
existence. For example, it is a matter of simple calculation to verify that if
$A$ is an atypical field then $R\left(  X,Y\right)  A=0$ for all $X,Y$ and
from here it follows that $Ric\left(  X,A\right)  =0$ for all $X$, therefore
the Ricci tensor is degenerate. So if $Ric_{p}$ is non-degenerate for a
certain point $p\in M$, then an atypical field cannot exist around $p.$

\end{document}